\documentclass[preprint,12pt]{elsarticle}
\usepackage{lineno,hyperref}
\modulolinenumbers[90]
\usepackage[cp1251]{inputenc}
\usepackage{fancybox}
\usepackage{pdfpages}

\usepackage[T2A]{fontenc}
\usepackage{xcolor}
\usepackage{amssymb,amsthm,amsmath,amsfonts,array,amscd}
\usepackage{geometry}
\usepackage{graphics}
\usepackage{lineno}
\usepackage{amsmath,amsthm,amsfonts,amssymb,amscd}
\usepackage[english]{babel}
\usepackage{mathrsfs}
\usepackage{url,graphics}

\begin{document}
\newtheorem{defin}{Definition}
\newtheorem{theorem}{Theorem}[section]
\newtheorem*{theorem*}{Theorem}
\newtheorem{sled}{Corollary}[section]
\newtheorem*{sled*}{Corollary}
\newtheorem{lemma}{Lemma}[section]
\theoremstyle{remark}
\newtheorem{rem}{\bf Remark}[section]
\newtheorem{ex}{\bf Example}[section]
\renewcommand{\baselinestretch}{1.3}
\newcommand{\tc}{\text{,}}
\newcommand{\tp}{\text{.}}
\date{ }
\makeatletter
\newcommand*{\ralph}[1]{\@ralph{\@nameuse{c@#1}}}
\newcommand*{\@ralph}[1]%
{\ifcase #1\or а\or б\or в\or г\or д\or е\or ж\or з\or и\or
	к\or л\or м\or н\or о\or п\or р\or с\or т\or у\or
	ф\or х\or ц\or ч\or ш\or щ\or э\or ю\or
	я\else\@ctrerr \fi}

\begin{frontmatter}
\title{{\bf  Generalized Limit Theorems For $U$-max Statistics}}

\bigskip

\author[a]{Nikitin Ya. Yu. \fnref{fn1}}

\author[a,b]{\,Simarova E. N. \fnref{fn2}}

\address[a]{ Department of Mathematics and Mechanics, Saint-Petersburg State University,\\ \phantom{$^1$} 
  Universitetsky pr. 28,Stary Peterhof 198504, Russia}

\address[b]{ Leonhard Euler International Mathematical Institute (SPbU Department),\\ \phantom{$^1$ }14th Line 29B, Vasilyevsky Island, St. Petersburg, 199178,
Russia}

\ead{katerina.1.14@mail.ru}

\fntext[fn1]{Research of  the author supported by joint grant RFBR-DFG No. 20-51-12004.}
\fntext[fn2]{Research of the author supported by Ministry of Science and Higher Education of the Russian Federation, agreement No. 075-15-2019-1619.}

\begin{abstract}
{\small
U-max statistics were introduced by Lao and Mayer in 2008. Instead of averaging the kernel over all possible subsets of the original sample, they considered the maximum of the kernel. Such statistics are natural in stochastic geometry. Examples are the maximal perimeters and areas of polygons and polyhedra formed by random points on a circle, ellipse, etc. The main method to study  limit theorems for  U-max statistics is a Poisson approximation. In this paper we consider a general class of kernels defined on a circle, and we prove a universal limit theorem with the Weibull distribution as a limit. Its parameters depend on the degree of the kernel, the structure of its  points of  maximum and the Hessians of the kernel at these points. Almost all limit theorems  known so far may be obtained as  simple special cases of our general theorem. We also consider several new examples. Moreover, we consider not only the uniform distribution of points but also almost arbitrary distribution on a circle satisfying mild additional conditions.}
\end{abstract}

\vskip10pt
\begin{keyword}
 Weibull distribution\sep Poisson approximation\sep $U$-max statistics\sep random perimeter\sep random area.

\MSC[2010] 60D05 \sep 60F05 \sep 60G70.
\end{keyword}
\vskip25pt
\end{frontmatter}

\section{Introduction}

$U$-statistics were introduced in probability by Halmos \cite{n1} and Hoeffding \cite{n2} in the mid-40s as a functional generalization of  sample mean. Let $ \xi_1, \xi_2, \dots $  be a sequence of independent identically distributed random elements taking values in a measurable space $(\mathfrak{X}, \mathfrak{A})$.  We define a real-valued symmetric Borel function $ f(x_1, \dots, x_m) $ on the space $ {\mathfrak{X}}^m, $ which we call a kernel of degree $ m. $

  $ U$-statistics are defined as follows:
\begin{equation}
\label{def}
U_n = {\binom{n}{m}}^{-1}\sum\limits_J f(\xi_{i_1}, \ldots, \xi_{i_m}),
\end{equation}
where $n \ge m$ and the set $$J = \{(i_1, \ldots , i_m) : 1 \le i_1 < \ldots < i_m \le n\}.$$
 Over the past decades, $U$-statistics have been studied in detail in many publications, and the state of the art is presented in monographs \cite{n3} and \cite{n4}.

In 2008 Lao and Meyer \cite{n5}, \cite{n6}, \cite{n9} independently considered the so-called $U$-$\max$ statistics which are obtained from (\ref{def}) as follows: the normalizing factor is removed and the sum is replaced by the maximum over the  set  $J:$
\begin{equation}
\label{def2}
H_n = \max\limits_J f(\xi_{i_1}, \ldots, \xi_{i_m}).
\end{equation}
 The $U$-$\min$ statistics are defined in a similar way. Such statistics often arise in stochastic geometry. The classical $U$-statistics are also used there, see, e.g., \cite{LR}.

Lao and Meyer studied mainly the limit behavior of maximal and minimal distances, areas and perimeters of figures formed by random points on a circle and sphere. The  geometric figure is determined by the kernel $ f. $ They used the Poisson approximation from the monograph \cite{n7} and the paper \cite{n8}, and proved a number of theorems on convergence to a Weibull distribution.

Here is a typical example from \cite{n5}, which gives an impression about such  limit theorems.

\begin{theorem} \label{theorem_A} Let $U_1, U_2, \ldots$ be independent and
uniformly distributed points on the unit circle
$S^1,\, \textup{peri}(U_i,U_j,U_l)$ be the perimeter of  triangle
formed by a triple of points $U_i,U_j,U_l$, \,
$H_n=\max\limits_{1\leq i<j<l\leq n}\textup{peri}(U_i,U_j,U_l).$  Then for any $t>0$
$$\lim\limits_{n\rightarrow\infty}\mathbb{P}\{n^3(3\sqrt{3}-H_n)\leq
t\}=1-\exp\left\{-\frac{2t}{9\pi}\right\}.$$
\end{theorem}
We note that among all triangles inscribed in the unit circle, the regular triangle has the maximal value of perimeter equal to $3\sqrt{3},$  see, e.g., \cite{n17}. It is clear that the maximal perimeter of a random triangle (which we denoted by $H_n$) tends to this value. The theorem indicates the normalization  necessary for this convergence and describes the limit distribution.

\medskip

Lao and Mayer  studied the kernels of low degrees, e.g.,  the area and perimeter of random triangles.  Koroleva and Nikitin considered $U$-$\max$ statistics of more complicated nature (see \cite{n14}). In particular, they considered the maximal perimeter among all  perimeters of convex $ m$-gons, where random vertices  are chosen from $ n $ independent points uniformly distributed on a circle. This was generalized in another direction in the papers \cite{n21} and \cite{n22}, where a generalized perimeter of  random convex polygon was considered.

In all these papers, the uniform distribution of  points on a unit circle was considered.
More general distributions were used by Lao and Mayer for some particular two-dimensional kernels, namely for the distances between points and the scalar product of two position vectors in \cite{n5}, \cite{n6}, \cite{n9}.  More complicated  kernels, namely, the areas and perimeters of inscribed polygons with more general conditions on  the distribution of  vertices were studied in \cite{n20}. A common feature of all papers    was that there were  investigated specific particular cases of $ U$-$\max$ kernels and $ U$-min statistics.

This paper is devoted to a significant generalization of known limit theorems for $U$-$\max$ and $U$-$\min$ statistics. We consider an almost arbitrary distribution of points on a circle, as well as a wide and general class of smooth kernels   with a natural structure of the set of extreme points. In this formulation, the limit behavior is determined by the degree of the kernel $ f, $ by the Hessian of the kernel at the maximal points  and also by the distribution of the points on the circle. The general formulas are used for kernels of  special type with convexity properties.
Most of the known before results can be deduced from the general theorem  but we also provide some new examples.

\section{Structure of the paper}

Lao and Mayer applied the Poisson approximation from the monograph \cite{n7}. This is still the main research method in this field. The following theorem plays a key role in studying the limit behavior of $U$-$\max$ statistics.

\begin{theorem}\cite{n7}. \label{LM}
Let $ \xi_1, \xi_2, \dots, \xi_n $  be a sequence of independent identically distributed random elements taking values in a measurable space $(\mathfrak{X}, \mathfrak{A})$ and function $f(x_1, \dots, x_m)$ be a real-valued symmetric Borel function, $f :{\mathfrak{X}}^m\rightarrow\mathbb{R}.$
Let $H_n=\max_J h(\xi_{i_1}, \ldots, \xi_{i_m})$ be the $U$-$\max$ statistics introduced in \eqref{def2} and define for any $z \in \mathbb{R}$  the following quantities:
\begin{align*}
p_{n,z}&=\mathbb{P}\{f(\xi_1,\ldots,\xi_m)>z\}, \, \,\,
\lambda_{n,z}={ n \choose m} p_{n,z},\\
\tau_{n,z}(r)&=\frac{\mathbb{P}\{f(\xi_1, \ldots,\xi_m)>z, f(\xi_{1+m-r},\xi_{2+m-r}, \ldots, \xi_{2m-r})>z\}}{p_{n,z}}.
\end{align*}

Then for all $n \ge m$ and for each $z \in \mathbb{R}$  we have
\begin{align}
&\nonumber|\mathbb{P}(H_n \le z)-e^{-\lambda_{n,z}}| \\
&\le \left( 1- e^{-\lambda_{n,z}} \right) \cdot \left[ p_{n,z}\left({ n \choose m} - { n-m \choose m} \right)+\sum_{r=1}^{m-1}{ m \choose r} { n-m \choose m-r} \tau_{n,z}(r) \right].
\label{lm1}
\end{align}
\end{theorem}

\begin{rem}\cite{n7}
\label{zz2}
 If the sample size $n$ tends to infinity, then the right-hand side in \eqref{lm1} is of asymptotic order
$$O\left(p_{n,z}n^{m-1}+
\sum_{r=1}^{m-1}\tau_{n,z}(r)n^{m-r}\right),$$ where for $m>1$  the first term is negligibly small with respect  to the sum.
\end{rem}

Silverman and Brown  \cite{n8}  have found the conditions for a general theorem used in \cite{n7} yielding a non-trivial Weibull law in the limit.

\begin{theorem} \cite{n8}
\label{SB} Let the conditions of  Theorem \textup{\ref{LM}} be satisfied. If, for some sequence of transformations $z_n: T \rightarrow \mathbb{R},\, T \subset \mathbb{R},$  the following equalities:
\begin{align}
\label{fff1}
&\lim_{n \rightarrow \infty}\lambda_{n, z_n(t)} = \lambda_t >0,\\
\label{fff2}
 &\lim_{n \rightarrow \infty} n^{2m-1} p_{n,z_n(t)}\tau_{n,z_n(t)}(m-1)=0
\end{align}
 hold for each $t \in T$, then
\begin{align}
\label{fff3}\lim_{n \rightarrow \infty} \mathbb{P}\left(H_n \le z_n(t)\right)=e^{-\lambda_t}
\end{align}
for all $t \in T.$
\end{theorem}
\begin{rem} \cite{n7}.
\label{rem2}
Condition \eqref{fff1} implies $p_{n,z}=O(n^{-m}).$ Therefore according to Remark \ref{zz2}   the rate of convergence in \eqref{fff3} is $$O\left(n^{-1}+\sum_{r=1}^{m-1}n^{2m-r}p_{n,z}\tau_{n,z}(r)\right).$$
Hence, for $m \ge 2$  condition \eqref{fff2} can be replaced by
\begin{align}
\label{zz1}
\lim_{n \rightarrow \infty}n^{2m-r}p_{n,z}\tau_{n,z}(r)=0 \text { for any } r \in \{1, \ldots, m-1\}.
\end{align}
\end{rem}

We will use the above assertions to prove our main results about the limit behavior of $U$-$\max$ statistics. Our paper consists of several parts. First, in Section \ref{tt1}, we will introduce the basic notation and restrictions, and then in Section \ref{form} we will formulate the general limit relations we have obtained.

Further, in Sections \ref{edg1} and \ref{edg2} we apply our results to the some specific classes of kernels. For these classes,  more explicit and relatively simple
limit theorems for $ U$-$\max$ statistics of geometric nature will be obtained. Some interesting examples will also be given.  A detailed proof of the main result of Section \ref{form} is rather painstaking and is of considerable length. Therefore, we placed it at the end of the paper; it occupies Sections \ref{doc} and  \ref{cont}.

\section{Preliminaries}
\label{tt1}

In this section we introduce necessary conditions and definitions.  We consider $ U$-$\max$ statistics with a fixed kernel $f$ depending on a set of $m $ points $ U_1,\ldots, U_m $ lying on the unit circle $ S^1 $ with center $ O $ at the origin, i.e.,
$$
f:\left(S^1\right)^m \rightarrow \mathbb{R} \cup \{-\infty\}.
$$

Denote by $ \beta_i $ the angle between the vectors $ OU_1 $ and  $ OU_{i + 1} $ (taken counterclockwise). We  call such angles central. In this way,
\begin{align}
\label{bb1}
\beta_i=\angle U_{1}OU_{i+1}.
\end{align}

Sometimes for the sake of brevity we will use the notation \begin{align}
\label{bet1}
\beta=(\beta_1, \ldots, \beta_{m-1}) \in [0, 2\pi)^{m-1}.
\end{align}

All angles that appear in this paper are considered modulo $ 2 \pi. $ All algebraic operations involving several angles are also considered modulo $ 2 \pi, $  unless otherwise stated.

Now we give some conditions that will be used  in the sequel.

{\bf A.  Conditions on the kernel $ f $.}
\label{tt2}

{\bf A1.} \ Function $ f $ is invariant with respect to rotations. Equivalently, this means that  function $ f $ can be written in the form
 $$
f(U_1, \ldots, U_m)= h(\beta_1, \ldots, \beta_{m-1})=h(\beta),
$$
where $\beta_i$ are  central angles, and $h$ is a function   $$h:[0,2\pi)^{m-1} \rightarrow \mathbb{R} \cup \{-\infty\}.$$

{\bf A2.} \ Function $ f $ cannot be changed after any permutation of the points $ U_1, \ldots, U_m. $ Therefore, the function $ h $ is also a symmetrical function of its arguments.

{\bf A3.} \ Function $ h $ is continuous and can be continuously extended to a function $h:[0,2\pi]^{m-1} \rightarrow \mathbb{R} \cup \{-\infty\}.$

{\bf A4.} \ Function $ h $ reaches its maximal value $ M $ and this maximum is realized only at a finite number of points $ V_1, \ldots, V_k \in [0, 2 \pi]^{m-1}. $ It is assumed that all these points do not lie on the boundary of the domain of definition of  function $ h $. In other words, $ V^j_i \in (0, 2 \pi) \text{ for all } \,  \, i \in \{1, \ldots, k\}, j \in \{1, \ldots, m-1\}, $ where $ V ^ j_i $ is the $ j$-th component of the point $ V_i. $

Condition A4 together with Condition A2 allows us to make the following conclusion about the structure of the  points of maximum of  function $ f $: there is only a finite number of points (up to rotations) where the maximal value of function $f$ is attained.  Moreover, all these points do not have matching components.

{\bf A5.} \ There exists $ \delta> 0 $ such that  function $ h $ is three times continuously differentiable in the $ \delta$-neighborhood of any maximum point $ V_i, $ \, $ i \in \{1, \ldots, k \} $.

{\bf A6.} \  Consider the Hessian matrix $ G_i $ of the  form: $$G_i=
\begin{pmatrix}
\frac{\partial^2 h(V_i)}{\partial^2 x_1} & \frac{\partial^2 h(V_i)}{\partial x_1 \partial x_2 } & \ldots & \frac{\partial^2 h(V_i)}{\partial x_1 \partial x_{m-1}} \\
\frac{\partial^2 h(V_i)}{\partial x_1 \partial x_2} & \frac{\partial^2 h(V_i)}{\partial^2 x_2 } & \ldots & \frac{\partial^2 h(V_i)}{\partial x_2 \partial x_{m-1}} \\
\vdots & \vdots & \ddots & \vdots \\
\frac{\partial^2 h(V_i)}{\partial x_{m-1} \partial x_1} & \frac{\partial^2 h(V_i)}{ \partial  x_{m-1}\partial x_2 } & \ldots & \frac{\partial^2 h(V_i)}{\partial^2 x_{m-1}} \\
\end{pmatrix}
.
 $$

We require that for all $ i \in \{1, \ldots, k \} $ the condition $$ \det(G_i) \ne 0 $$ holds.

\medskip

{\bf B. Conditions on the distribution of points.}
\label{rt}

 {\bf B1.} \ The random points $ U_1, \ldots, U_n $ are independently distributed on the unit circle $ S^1 $ with the same probability density $ p(x). $

 {\bf B2.} \ The density $p$ is continuous (therefore, it can be considered as a non-negative continuous $ 2 \pi $-periodic function  $p: \mathbb{R} \rightarrow \mathbb{R_{+}}$ such that $ \int \limits_0^{2 \pi} p(x) = 1 $).

{\bf B3.} \ There exists at least one maximal point of the kernel (which we denote by $V_*$) such that
$$\int\limits_{0}^{2\pi} \left[p(x)\prod\limits_{l=1}^{m-1} p(x+V_{*}^l)\right] \, dx \ne 0.$$

Similar conditions on $p(x)$ arose in  \cite{n20}.

\begin{rem}
The conditions imposed on the density $p$ are not too restrictive. For example, continuous densities separated from zero or continuous densities taking the value 0 only on a set of measure less than $ \frac{2 \pi}{m} $ are suitable for these conditions. A useful example of non-uniform distribution, is the von Mises distribution (see, e.g., \cite{DS}).
\end{rem}

\section{Main results}
\label{form}

Now we state the main result of this paper.

\begin{theorem}[\bf General theorem]
\label{t2}
Suppose that  kernel $f$ and  points $U_1, \ldots, U_n$ satisfy all  Conditions  \textup{A} and \textup{B.}
Let $H_n$ be the $U$-$\max$ statistics constructed by kernel $f$, that is,
$ H_n = \max\limits_{1 \le i_1< \ldots <i_m \le n} f(U_{i_1}, \ldots, U_{i_m}).$

Then for every  $t>0$ the following relation holds true:
\begin{align}
\label{main}
\lim_{n \rightarrow \infty} \mathbb{P}\lbrace n^{\frac{2m}{m-1}} (M-H_n)\le t\rbrace=1 -e^{-\frac{t^{\frac{m-1}{2}}K}{m!}},
\end{align}
where
$K=\frac{\left(2\pi\right)^{\frac{m-1}{2}}}{\Gamma\left(\frac{m+1}{2}\right)}\sum\limits_{i=1}^k \left( \frac{1}{\sqrt{\det(-G_i)}} \int\limits_0^{2\pi} p(x) \prod\limits_{l=1}^{m-1} p(x+V_i^l) \, dx \right) $ and $M$ is from Condition \textup{A4.}
The rate of convergence is $O\left(n^{-\frac{1}{m-1}}\right)$ for  $m>1$ and   $O(n^{-1})$
for $m=1$.
\end{theorem}

\medskip

Theorem \ref{t2} immediately implies several simple but very useful consequences. As far as we know, these consequences are new.
First of all, Theorem \ref{t2} can be modified slightly for $ U$-$\min$ statistics.
\begin{sled}
Let us replace the maximum  $M$ with the minimum  $ \mu $ and  consider the points of minimum in  Conditions \textup{A4, A5, A6} and \textup{B3.} Denote by $H_n$  the $U$-$\min$ statistics constructed by kernel $f$, that is,
$ H_n = \min\limits_{1 \le i_1< \ldots <i_m \le n} f(U_{i_1}, \ldots, U_{i_m}).$

 Then for each  $t>0$ we have
$$\lim_{n \rightarrow \infty} \mathbb{P}\{ n^{\frac{2m}{m-1}} (H_n-\mu)\le t \}=1 -e^{-\frac{t^{\frac{m-1}{2}}K}{m!}},$$ where
$K=\frac{\left(2\pi\right)^{\frac{m-1}{2}}}{\Gamma\left(\frac{m+1}{2}\right)}\sum\limits_{i=1}^k\left( \frac{1}{\sqrt{\det(G_i)}} \int\limits_0^{2\pi} p(x) \prod\limits_{l=1}^{m-1} p(x+V_i^l) \, dx\right) .$
\end{sled}

\medskip

\begin{sled}[\bf Uniform distribution]
Suppose that $ U_1, \ldots, U_n $ are independently and uniformly distributed points on the unit circle. Let $ H_n $ be $U$-$\max$ statistics with a kernel $ f $ satisfying  Conditions \textup{A} and \textup{B.}
Then for any $ t>0 $ we have $$\lim_{n \rightarrow \infty} \mathbb{P}\lbrace n^{\frac{2m}{m-1}} (M-H_n)\le t\rbrace=1 -e^{-\frac{t^{\frac{m-1}{2}}K}{m!}},$$ where $K=\frac{1}{\left(2\pi\right)^{\frac{m-1}{2}}\Gamma\left(\frac{m+1}{2}\right)} \sum_{i=1}^k \frac{1}{\sqrt{ \det(-G_i)}}.$
The rate of convergence is the same as in \eqref{main}.
\end{sled}

 Consider a function $ f $ satisfying Conditions A and B.
By Condition A4, the coordinates of all $ V_i $ do not coincide. Therefore, all the  points  of maximum  of the corresponding function $ h $ are divided into permutation orbits of length $ (m-1)!, $ and the determinant of the Hessian matrix of  function $ h $ at all these points is the same.
Denote by $ W_1, \ldots, W_r $  points of  maximum ordered by ascending of the angles.
Then we have another useful corollary  simplifying the formulations.
\begin{sled}
\label{z8}
Let $ H_n $ be $U$-$\max$ statistics with a kernel $ f $ satisfying  Conditions \textup{A} and \textup{B.}
Then
for any $ t>0 $ the following limit relation holds:
$$\lim_{n \rightarrow \infty} \mathbb{P}\lbrace n^{\frac{2m}{m-1}} (M-H_n)\le t \rbrace=1 -e^{-\frac{t^{\frac{m-1}{2}}K}{m}},$$ where
$K=\frac{\left(2\pi\right)^{\frac{m-1}{2}}}{\Gamma\left(\frac{m+1}{2}\right)}\sum\limits_{i=1}^r \frac{1}{\sqrt{\det(-G_i)}}\int\limits_0^{2\pi} \left[p(x) \prod\limits_{l=1}^{m-1} p(x + W_i^l) \right]\, dx $
 and $G_i$ is the Hessian matrix at the point $W_i$.
\end{sled}

\begin{rem}
Theorem \ref{t2} may be applied to a rather wide class of kernels from stochastic geometry but of course not to all of them.  In \cite{n22} the definition of  generalized perimeter was introduced.  It is the sum of the $ y$-th degrees of the side lengths of a polygon constructed on given $ m $ points on a circle. It is shown that for $ y>1 $ and $m > 1+\pi/\left(\arccos{\frac{1}{\sqrt{y}}}\right)$ the generalized perimeter attains its maximum on the configuration of points in which some of them  coincide. This shows that our result cannot be applied to this rather simple variant of the problem. In this case the limit relation is still an open question.
\end{rem}

In the following sections we  give limit theorems for $ U $-$\max$ statistics generated by  kernels of a special form that attain their maximum only at the vertices of a regular polygon. This  special case  covers a large number of standard geometric characteristics.

\section{Limit behavior of $U$-max statistics for several  functions depending on the side lengths of the polygon}
\label{edg1}
Consider a function  $f:	\left(S^1\right)^m \rightarrow \mathbb{R}$ which is invariant with respect to rotation. It may be defined as a function $h: [0, 2\pi)^{m-1} \rightarrow \mathbb{R},$  where the set of angles  $(\beta_1, \ldots, \beta_{m-1})$ is determined by the set of points $(U_1, \ldots, U_m) \in \left(S^1\right)^m$ according to  (\ref{bb1}). The connection between these functions is given by
\begin{align}
\label{f27}
f(U_1, \ldots, U_m) = h(\beta_1, \ldots, \beta_{m-1}).
\end{align}

Consider some function $g:[0, 2\pi] \rightarrow \mathbb{R}$ which is continuous and three times continuously differentiable in some neighborhood of the point $ \frac{2 \pi}{m}. $ We also require  $g''(\frac{2\pi}{m}) \ne 0.$ Define the functions $ f $ and $ h $ as follows:
 \begin{align}
\label{f22}
&f(U_1, \ldots, U_m)=h(\beta_1, \ldots, \beta_{m-1}) = \sum_{i=1}^{m}g(\beta_i-\beta_{i-1}),\\
&\nonumber \text{ where } 0=\beta_0 \le \beta_1 \le \ldots \le \beta_{m-1}\le  \beta_m=2\pi,
\end{align}
and further define the function $ h $ on $ [0, 2 \pi]^{m-1} $ so that it is symmetric, and corresponding to the function $ f. $
Defined in such a way  function $ f $  satisfies conditions A1 and A3. Condition A2 also holds since  function $ f $ depends only on the angles between neighboring   vectors $ OU_1, \ldots, OU_m. $

Under additional restrictions we get the following statement:

\begin{theorem}
\label{tt7}
Suppose that the points $ U_1, \ldots, U_n $ are independently distributed on $ S^1 $ with a common continuous density $ p(x) $ such that   $\int_{0}^{2\pi} \prod_{l=0}^{m-1} p(x+\frac{2\pi l}{m}) \, dx >0.$
Consider the $ U$-$\max$ statistics $ H_n $ with kernel  $ f$  of the form  \eqref{f22}. Suppose that this kernel attains its maximum only at the vertices of a regular $ m$-gon.

Then for any $ t>0 $ the following  limit relation holds:
\begin{align}
\label{f23}
\lim_{n \rightarrow \infty} \mathbb{P}\Big \lbrace n^{\frac{2m}{m-1}} \left(mg\left(\frac{2\pi}{m}\right)-H_n\right)\le t\Big \rbrace=1 -e^{-\frac{t^{\frac{m-1}{2}}K}{m}},
\end{align}
 where
$$K=\frac{\left(2\pi\right)^{\frac{m-1}{2}}\left[\int_{0}^{2\pi} \prod_{l=0}^{m-1} p(x+\frac{2\pi l}{m})\, dx \right]}{ \left(-g''\left(\frac{2\pi}{m}\right)\right)^{\frac{m-1}{2}}\Gamma\left(\frac{m+1}{2}\right)\sqrt{m}}.$$
\end{theorem}
\begin{proof}

Let us first prove that  function $ f $ and density $p $  satisfy Conditions A and B from Section 3. The statement of Conditions A1, A2, A3 was established above. The fulfillment of Condition A4 follows from the fact that a regular polygon is the only maximal point  of  function $ f. $  Condition A5 follows from  formula $ (\ref{f22}) $ and the differentiability assumption. Conditions B1 and B2  are obviously satisfied (they are assumed in the statement of the theorem). It remains to check the properties  A6 and B3,  then  we can use Theorem \ref{t2}.

We use the arguments from Corollary \ref{z8}. It allows us to restrict ourselves only on the case
 $0\le \beta_1\le \ldots\le \beta_{m-1}\le 2\pi.$ By  (\ref{bet1}), a regular $ m$-gon corresponds to some permutation of the angles from the set $V^*=(V^1, \ldots, V^{m-1}),$ where $ V^i = \frac{2 \pi i }{m} $. Thus, the condition that the maximum of  function $ f $ is attained only on the regular $ m $-gon means that the point $ V^* $ is the only point of  maximum of  function $ h $  among all points with  ordered  angles. Together with the condition $\int_{0}^{2\pi} \prod_{l=0}^{m-1} p(x+\frac{2\pi l}{m})\, dx >0$ it implies the validity of Condition B3.
 Next, we obtain an explicit formula for the determinant of the Hessian matrix of  function $ h $ at the point $ V^* $ and make sure that it is not equal to zero. This fact implies the fulfillment of Condition A6, and then an application of Corollary \ref{z8} finishes the proof of Theorem \ref{tt7}.

By simple calculations, we get
$$
\frac{\partial^2 h}{\partial x_i \partial x_j} \left(V^*\right) = \begin{cases}
0, \text { if } |i-j|>1,\\
2g''\left(\frac{2\pi}{m}\right), \text{ if } i=j,\\
-g''\left(\frac{2\pi}{m}\right), \text{ if } |i-j|=1.
\end{cases}
$$

Therefore, the Hessian matrix at the point $ V^* $ is
\begin{align*}
G(V^*)=g''\left(\frac{2\pi}{m}\right) \begin{pmatrix}
2 & -1 & 0 & 0 & \ldots & 0\\
-1 & 2 & -1 & 0 & \ldots & 0\\
0 & -1 & 2 & -1 & \ldots & 0\\
\vdots & \vdots & \vdots & \vdots & \ddots & \vdots \\
0 & 0 & 0 & 0 & \ldots & 2\\
\end{pmatrix}                             = g''\left(\frac{2\pi}{m}\right)B_{m -1}.
\end{align*}
 The determinant of the tridiagonal matrix $ B_ {m-1} $ can  be easily calculated
  using a recurrence relation. Denote by $ d(n) $ the determinant of such $ n \times n $  matrix.  Then it is easy to see that
$$\begin{cases}
d(1)=2,\\
d(2)=3,\\
d(n)=2d(n-1)-d(n-2).
\end{cases}
$$
Therefore, $ d(n) = n+  1. $
Hence,
$$
\det(-G(V^*))=m\cdot\left(-g''\left(\frac{2\pi}{m}\right)\right)^{m-1} \ne 0,
$$
 and the conditions of  Theorem \ref{t2} are satisfied. Substituting the determinant of  Hessian matrix into the formula from Corollary \ref{z8} we obtain the required limit relation.
\end{proof}

\begin{theorem}
\label{tt6}

For strictly concave functions  $g:[0,2\pi] \rightarrow \mathbb{R}$  the maximum of the function $ f $ defined in  \eqref{f22} is attained only at the vertices of the regular $ m $-gon.
\end{theorem}
\begin{proof}
We assume that $0 \le \beta_1 \le \ldots \le \beta_{m-1} \le 2\pi.$
Let us prove that the point $V^*=(V^1, \ldots, V^{m-1}),$ where $ V^i = \frac{2 \pi i}{m} $ is the only  point of maximum of the function $f$ among all points with  ordered  angles. Due to  Jensen's inequality we have
\begin{align}
\label{f19}
&f(U_1, \ldots, U_m) = \sum_{i=1}^{m}g\left(\beta_i-\beta_{i-1}\right) \ge m g\left(\frac{\beta_m-\beta_0}{m}\right)\\
&\nonumber =mg\left(\frac{2\pi}{m}\right) =\sum_{i=1}^{m}g\left(V^i-V^{i-1}\right).
\end{align}
 Function $ g $ is strictly convex; therefore, if not all arguments of the function $ g $ are equal to each other,  inequality in $ (\ref{f19}) $ is strict.

\end{proof}
By combining Theorems \ref{tt7} and \ref{tt6} we obtain the  following corollary.
\begin{sled}
\label{sl1}

Suppose that  function  $g:[0, 2\pi] \rightarrow \mathbb{R}$ is continuous, strictly concave, three times continuously differentiable in a neighborhood of the point $ \frac{2 \pi}{m} $  and $ g'' (\frac{2 \pi}{m}) \ne 0. $ Let  function $ f $be  defined in \eqref{f22}.  Then $ U $-$\max$ statistics with  kernel $ f $ satisfies  relation \eqref{f23} of Theorem \textup{\ref{tt7}}.
\end{sled}

\begin{sled}
\label{sl2}
In Theorem \textup{\ref{tt7}} we may consider functions $ f $ for which the regular polygon is the only  point of minimum.  Then for the $ U$-$\min$  statistics $ H_n $ generated by  kernel $ f $ the following limit relation holds for any   $t>0$ \textup{(}similarly to  \eqref{f23}\textup{)}:

$$\lim_{n \rightarrow \infty} \mathbb{P}\Big\lbrace n^{\frac{2m}{m-1}} \left(H_n-mg\left(\frac{2\pi}{m}\right)\right)\le t\Big\rbrace=1 -e^{-\frac{t^{\frac{m-1}{2}}K}{m}},$$ where
$$K=\frac{\left(2\pi\right)^{\frac{m-1}{2}}\left[\int_{0}^{2\pi} \prod_{l=0}^{m-1} p(x+\frac{2\pi l}{m})\, dx \right]}{ \left(g''\left(\frac{2\pi}{m}\right)\right)^{\frac{m-1}{2}}\Gamma\left(\frac{m+1}{2}\right)\sqrt{m}}.$$

In particular, this is the case if we consider the strictly convex functions g in  Corollary \textup{\ref{sl1}} instead of strictly concave ones.

\end{sled}

\begin{sled}
\label{rr2}

If $p$ is the  uniform density, then constant K from Theorem \textup{\ref{tt7}} satisfies
\begin{align}
\label{f21}
K=\frac{1}{ \left(-2\pi g''\left(\frac{2\pi}{m}\right)\right)^{\frac{m-1}{2}}\Gamma\left(\frac{m+1}{2}\right)\sqrt{m}}.
\end{align}
\end{sled}

\begin{ex}{\it Maximal perimeter of  inscribed polygon}
\label{exx1}

Let us consider the maximal perimeter of an inscribed convex $ m $-gon with random vertices on a circle.  These are $ U$-$\max$ statistics with  kernel $ f $ of the form (\ref{f22}), where $ g(x) = 2 \sin{\left(\frac{x}{2}\right)}. $
This function is strictly concave, therefore the results follow from Theorem \ref{tt7}. The limit behavior of $ U $-$\max$ statistics with such a kernel in the case of uniform distribution of points may be found in \cite{n14}, the result coincides with
Corollary  \ref{rr2} with this function $g,$ so that the result \cite{n14} is a simple special case of our results.
\end{ex}

\begin{ex}{\it Maximal area of  inscribed polygon.}
\label{exx2}

Another $ U $-$\max$ statistics considered in \cite{n14} was the area of   inscribed convex $ m $-gon. It is generated by kernel $ f $ of the form (\ref{f22}), where $ g(x) = \frac12 \sin{x}. $
The maximum of  function $ f $ is attained only on the regular $ m$-gon, see, for example, \cite[problem 57a]{n17}.
Therefore, the limit theorems for these $ U $-$\max$ statistics  follow directly from Theorem \ref{tt7} and Corollary \ref{rr2}. The results of \cite{n14} and  more general results of \cite{n20} again follow from ours as simple special cases. Similar statements hold for the areas and perimeters of the {\it described} random polygons considered in \cite{n14}, they also follow from Theorem \ref{tt7} with some $g(x)$.

\end{ex}

\begin{ex}{\it Sum of the distances from the center to the vertices of  described polygon}
\label{exx3}

Let us consider now an example of  kernel   not arising earlier in the literature on  the limit behavior of $ U $-$\max$ statistics. We define  kernel $f: \left(S^1\right)^m \rightarrow \mathbb{R} \cup \{+\infty\}$ as follows: construct the described convex $ m $-gon with vertices at points $ A_1, \ldots, A_m $ such that its sides touch the circle $ S^1 $ at points $ U_1, \ldots, U_m. $ Define the  function $$ f(U_1, \ldots, U_m) = \sum \limits_{i = 1}^m |OA_i| $$ as the sum of  distances from the center to the vertices of the described $ m $-gon. It is possible that some vertex $ A_i $ goes to infinity, in this case we define $ f(U_1, \ldots, U_m) = + \infty. $

Function $ f $ can be written in the form (\ref{f22}), where $g(x)=(\cos{(\frac{x}{2})})^{-1},$ if $ 0 \le x <\pi, $ and $ g(x) = + \infty $ otherwise. The case $ g(x) = + \infty $ is possible where some vertex $ A_i $ goes to infinity.
Note that
$$
g''(x)=\frac{1+\sin^2{(\frac{x}{2})}}{4\cos^3{(\frac{x}{2})}} >0 \quad  \text{ for }\quad x \in [0, \pi),
$$
therefore,  function is strictly convex on $  [0, \pi). $ By Corollary \ref{sl2}, the minimum will be attained at the vertices of the regular $m$-gon, and for $ U $-$\min$  statistics $ H_n $ generated by  kernel $ f $ we have  for any $ t> 0 $:
$$\lim_{n \rightarrow \infty} \mathbb{P}\Big\lbrace n^{\frac{2m}{m-1}} \left(H_n-\frac{m}{\cos{\frac{\pi}{m}}}\right)\le t \Big\rbrace=1 -e^{-\frac{t^{\frac{m-1}{2}}K}{m}},$$ where
$$K=\frac{\left(8\pi \cos^3{\frac{\pi}{m}} \right)^{\frac{m-1}{2}}\left[\int_{0}^{2\pi} \prod_{l=0}^{m-1} p(x+\frac{2\pi l}{m}) \, dx \right]}{ \left(1+\sin^2{\frac{\pi}{m}}\right)^{\frac{m-1}{2}}\Gamma\left(\frac{m+1}{2}\right)\sqrt{m}}.$$
\end{ex}

\begin{ex} {\it Generalized  perimeter of the polygon}
\label{exx6}

In \cite{n21}, \cite{n22} a  definition of  generalized perimeter was introduced. Generalized perimeter of order $ y $ is the sum of the $ y$-th degrees of the side lengths of a convex inscribed polygon constructed on given $ m $ points. This function may also be written as $f(U_1, \ldots, U_m)=\sum_{i=1}^mg(\beta_i-\beta_{i-1}),$ where $g(x)=2^y\sin^y(\frac{x}{2}).$ Function $ g $ is convex for negative $ y $ and concave for $ y \in (0,1]; $ therefore, the limit relation obtained in these cases for $ U $-$\min$ and $ U $-$\max$ statistics, respectively,  is also a special case of the equality (\ref{f21}) (see \cite{n21} ).

For $ y \in (1,2]$ and  $m = 3 $,  function $ g $ is not  concave but  function $ f $  also attains its maximum only on the vertices of  regular triangle, therefore  relation (\ref{f23}) with  constant  (\ref{f21}) from  Theorem \ref{tt7} also holds true in this case. This was shown in \cite{n22}.
\end{ex}

\begin{ex} {\it Further generalization of the perimeter}
\label{exx4}

The concept of  generalized perimeter introduced in \cite{n21} may be generalized further.
Suppose that  kernel $ f $ is given by
\begin{align}
\label{f25}
f(U_1, \ldots, 	U_m)=\sum_{i=1}^m r(|U_iU_{i+1}|)
\end{align}
where $r$ is a function$\colon [0,2] \rightarrow \mathbb{R}$  and the points $ U_1, \ldots, U_m $  are ordered counterclockwise, and also assume that  kernel  is symmetric. By $|U_{i}U_{i+1}|$ we denote the length of the side of the polygon.
 The generalized perimeter introduced in \cite{n21} and \cite{n22} corresponds to  function $ r (x) = x^y. $

If  function $r$ is continuous, strictly concave, increasing,   three times continuously differentiable in some neighborhood of the point $2\sin{\frac{\pi}{m}}$ and $r''(2\sin{\frac{\pi}{m}}) \ne 0,$ then  function $g(x) = r(2\sin{\frac{x}{2}})$  is  strictly convex. By Сorollary \ref{sl1}, we can write the limit relation for $ U $-$\max$ statistics from Theorem \ref{tt7}. It is also possible to replace the conditions of strict concavity and increasing  by the conditions of concavity and strict increasing.

Similarly, if $ r $ is a continuous strictly convex decreasing function which is three times continuously differentiable in some  neighborhood of the point $2\sin{\frac{\pi}{m}},$ and $r''(2\sin{\frac{\pi}{m}}) \ne 0,$ then the function $g(x) = r(2\sin{\frac{x}{2}})$  is  strictly convex. By Сorollary \ref{sl2}, we can write the limit relation for $ U $-$\min$ statistics. The condition of strict convexity and decreasing may be replaced by the condition of convexity and strict decreasing.
\end{ex}

\begin{ex}
\label{exx5}
Let us use  function $r(x)=e^{-ax}x^b (\ln{\left(\frac{x}{2}\right)})^c$ in Example \ref{exx4}. Such functions were considered  by Alexander and Stolarsky in \cite{n24}.
Denote by   $ \tau(a, b, c) $ the function which is equal to 1 if function $ r $ is strictly concave and increasing, and  is equal to $-1$ if function $ r $ is strictly convex and decreasing.
Then the following equality from \cite{n24} is true:
\begin{align*}
\tau(a,b,c)=\begin{cases}
(-1)^c, &\text{ if  } a \ge 0, b \le 0, c \in \mathbb{N},\\
-1, &\text{ if  } a \ge 0, b \le 0, c=0, a^2+b^2 \ne 0,\\
1, &\text{ if  } a=0, 0<b \le 1, c=0,
\end{cases}
\end{align*}
Using the arguments from the previous example, we may study the limit behavior of $ U $-$\max$ (respectively $ U $-$\min$) statistics $ H_n $ with  kernel $ f, $ constructed by  $ (\ref{f25}), $ in the case where $ \tau = 1 \,$ (respectively $\tau = -1$).  It can be simply done using $ g(x) = r (2 \sin{\frac{x}{2}} )$ in  Theorem \ref{tt7}.
  \end{ex}

\section{Limit behavior of $ U $-max statistics for several   functions depending on the side lengths and diagonals of a polygon}
\label{edg2}

In this section, the arguments are very similar to  those in Section  \ref{edg1}. We define  kernel $ f $ and  corresponding function $ h $ by the set of angles $(\beta_1, \ldots\beta_{m-1})$  introduced in (\ref{bet1}).   The connection between   $ f $ and $ h $ is  given by   (\ref{f27}). We again define the function $ h $ using a continuous function $ g: [0, 2 \pi] \rightarrow \mathbb{R}$ but in a different way. The restrictions  on  function $ g $ will be different, and the function $ h $ itself is defined via the following analogue of   (\ref{f22}):
\begin{align}
\label{f24}
f(U_1, \ldots, U_m) =h(\beta_1, \ldots, \beta_{m-1})  = \sum_{0 \le i<j \le m-1}g(|\beta_j-\beta_{i}|), \text{ where } \beta_0=0.
\end{align}

In other words, in this section we consider {\it functions depending on the angles between any pairs of points $ U_i $ and $ U_j, $ but not only on the angles between adjacent points,} as was done in the last section.

Let us describe the extreme points of  function $ f $ for concave $ g $ having some symmetry property. This is an analogue of Theorem \ref{tt6}.

\begin{theorem}
\label{tt9}
Suppose that  function $ g: [0, 2 \pi] \rightarrow \mathbb{R}$ is continuous strictly concave function such that $ g(x) = g (2\pi-x). $ Then  function $ f $ defined in  \eqref{f24} attains its maximum only at the vertices of  regular $ m $-gon and its maximal value is equal to  $ \frac{1}{2} \sum_ {s = 1}^{m-1} mg \left(\frac{2 \pi s}{m} \right). $

\end{theorem}
\begin{proof}

Without loss of generality, we assume that the vertices $ U_1, \ldots, U_m $ are ordered counterclockwise.
Denote by
$P(k, U_1, \ldots, U_m)$ the sum  $\sum_{i=1}^{m} g(\beta_{i+k}-\beta_i)$  assuming that $\beta_0=0, \beta_{m+s}=2\pi+\beta_s$ for $ s \in \{0, \ldots, m-1 \}. $
Then $$f(U_1, \ldots, U_m)=\frac{1}{2}\sum_{i=1}^{m-1} P(i,U_1, \ldots, U_m).$$
We prove that the maximum of $ P (k, U_1, \ldots, U_m) $ is attained only on the regular $ m$-gon.

Due to strict concavity
\begin{align*}
P(k, U_1, \ldots, U_m)=\sum_{i=1}^{m} g(\beta_{i+k}-\beta_i) \le m g\left(\frac{2\pi k}{m}\right),
\end{align*}
and  equality is achieved only when $ \beta_{i + k} - \beta_i = \frac{2 \pi k}{m} $ for all $i$ and $k.$ Therefore, the maximal value is attained only at the vertices of the regular polygon.

\end{proof}

Now we state an analogue of Theorem \ref{tt7} for  kernel $ f $ of the form (\ref{f24}).

\begin{theorem}
\label{tt8}
Suppose that the points $ U_1, \ldots, U_n $ are independently distributed on $ S^1 $ with a common continuous density $ p(x) $ such that  $\int_{0}^{2\pi} \prod_{l=0}^{m-1} p(x+\frac{2\pi l}{m})\, dx >0.$
Consider a continuous function $g\colon[0, 2\pi] \rightarrow \mathbb{R}$ which is  three times continuously differentiable in the neighborhoods of points $ \frac{2 \pi s}{m} $ for all $ s \in \{1, \ldots, m-1 \}$ and such that $ g(x) = g(2 \pi-x). $
We construct  function $ f $  by equality \eqref{f24} and suppose that it attains its maximum only at the vertices of a regular $ m$-gon. Let $ H_n $ be $ U $-$\max$ statistics with  kernel $ f. $
Consider the  symmetric matrix  $G=(g_{i,j})_{i,j=1}^{m-1},$ where
 \begin{align}
 \label{f29}
g_{i,j}=\begin{cases}
-g''\left(\frac{2\pi|i-j|}{m}\right), \text{ if } i \ne j,\\
\sum\limits_{s=1}^{m-1}g''\left(\frac{2\pi s}{m}\right), \text{ if } i=j.
\end{cases}
 \end{align}
If $\det{G}\ne 0,$ then, for any $ t>0, $ the following limit relation holds:
\begin{align}
\label{f26}
\lim_{n \rightarrow \infty} \mathbb{P}\Big\lbrace n^{\frac{2m}{m-1}} \left(\frac{1}{2}\sum_{s=1}^{m-1}mg\left(\frac{2\pi s}{m}\right)-H_n\right)\le t\Big\rbrace=1 -e^{-\frac{t^{\frac{m-1}{2}}K}{m}},
\end{align}
 where
$$K=\frac{\left(2\pi\right)^{\frac{m-1}{2}}\left[\int_{0}^{2\pi} \prod_{l=0}^{m-1} p(x+\frac{2\pi l}{m})\, dx \right]}{ \sqrt{\det(-G)}\Gamma\left(\frac{m+1}{2}\right)}.$$
\end{theorem}
The proof of this theorem is similar to the proof of Theorem \ref{tt7} and therefore omitted.

\begin{rem}
\label{rr3}
If the values of  second derivatives at the points $ \frac {2 \pi s}{m} $ are negative for all $ s \in \{1, \ldots, m-1 \}, $ then $g_{i,i}>\sum_{j=1 \\i \ne j}^{m-1} |g_{i,j}|.$ Matrices with such a property  are called   diagonally dominant and    according to \cite[Ch. 6, \S 1,  p. 392, Th. 6.1.10]{MA} the determinants of such matrices are non-zero.
\end{rem}
\begin{rem}
The matrix $ G $ is a Toeplitz matrix.
\end{rem}

Combining the results of Theorems \ref{tt9} and \ref{tt8}, we obtain the following Corollary:
\begin{sled}
If  function $g: [0, 2\pi] \rightarrow \mathbb{R}$ is continuous, strictly concave,  three times continuously differentiable in the  neighborhoods of points $ \frac{2 \pi s}{m} $ for all $ s \in \{1, \ldots, m-1 \}, $ and also has the property $ g(x) = g(2 \pi-x), $  and  if probability density $p$ is continuous and $\int_{0}^{2\pi} \prod_{l=0}^{m-1} p(x+\frac{2\pi l}{m})\, dx >0,$   then for $ U $-$\max$ statistics $ H_n $ with kernel $ f $ of the form \eqref{f24}  relation \eqref{f26} of Theorem \textup{\ref{tt8}} holds.

\end{sled}
Similarly to Section \ref{edg1}, this statement can be reformulated for $ U $-$\min$ statistics and strictly convex functions $ g. $

\begin{rem}
One of the main differences between Sections \ref{edg2} and \ref{edg1} is the condition $ g(x) = g (2 \pi-x),$ which  was not involved in Section \ref{edg1}. It gives the equality of  elements on the main diagonal of the matrix $ G $ introduced in Theorem \ref{tt8}. This condition also implies that  $ g (| \beta_i- \beta_j |) = r (| U_iU_ {j} |), $ where $ | U_iU_{j} | $ is the length of the segment $ U_iU_j. $

Thus, by analogy with Example \ref{exx4}, we may say that all the functions $ f $ satisfying (\ref{f24}) are the generalized sums of pairwise distances between points. This shows that functions $g$ from  Examples \ref{exx2} and \ref{exx3} cannot be used for constructing  functions $ f $ in (\ref{f24}).
This property holds for  function $g$ from Examples \ref{exx1}, \ref{exx6}, \ref{exx4}, \ref{exx5}.  Such functions $g$ are also differentiable on $ (0, 2 \pi), $ and Remark \ref{rr3} holds for them. Therefore,  functions $ f $ constructed in these examples by  (\ref{f24})  satisfy Theorem \ref{tt8}.
\end{rem}

\begin{ex}
As already mentioned,  function $ g(x) $  introduced in Example \ref{exx5} may be used in this case as well.
\end{ex}

\begin{ex}{\it Sum of pairwise distances between vertices}
\label{exx8}

Consider the simple case where $g(x)$ is defined in Example \ref{exx1}, and the resulting function $ f (U_1, \ldots, U_m) $ is the sum of the pairwise distances between  points $ U_1, \ldots, U_m . $
According to paper \cite{n25}, the maximum is attained only at the vertices of the regular polygon and is equal to $ m \cot{\frac{\pi}{2m}}.$
The limit relation \eqref{f26} holds but  we did not manage to calculate the exact constant $K$ in it for an arbitrary $ m.$
The matrix $ -G $ involved in Theorem \ref{tt8} has the following form:
$$\begin{pmatrix}
\frac{1}{2} \cot{ \frac{\pi}{2m}} & -\frac{1}{2} \sin{\left(\frac{\pi }{m}\right)} & -\frac{1}{2} \sin{\left(\frac{2\pi }{m}\right)} &  \ldots & -\frac{1}{2} \sin{\left(\frac{(m-2)\pi }{m}\right)}\\
-\frac{1}{2} \sin{\left(\frac{\pi }{m}\right)}& \frac{1}{2} \cot{ \frac{\pi}{2m}} & -\frac{1}{2} \sin{\left(\frac{\pi }{m}\right)} &  \ldots & -\frac{1}{2} \sin{\left(\frac{(m-3)\pi }{m}\right)}\\
-\frac{1}{2} \sin{\left(\frac{2\pi }{m}\right)}& -\frac{1}{2} \sin{\left(\frac{\pi }{m}\right)}&\frac{1}{2} \cot{ \frac{\pi}{2m}} &   \ldots & -\frac{1}{2} \sin{\left(\frac{(m-4)\pi }{m}\right)}\\
\vdots & \vdots & \vdots  & \ddots &\vdots\\
 -\frac{1}{2} \sin{\left(\frac{(m-2)\pi }{m}\right)} &  -\frac{1}{2} \sin{\left(\frac{(m-3)\pi }{m}\right)} &  -\frac{1}{2} \sin{\left(\frac{(m-4)\pi }{m}\right)}&   \ldots & \frac{1}{2} \cot{ \frac{\pi}{2m}}\\
\end{pmatrix}.
$$
This is a Toeplitz matrix, but we could not calculate its determinant in a general form. However, we calculated it for the small values of $m$:
\begin{align*}
& \det(-G)=\frac{9}{16} =0.5625  \text{ for } m=3;\\
& \det(-G)=\frac{3\sqrt{2}+4}{8}\approx 1.03033 \text{ for } m=4;\\
&  \det(-G)=\frac{175+75\sqrt{5}}{128}\approx 2.6773835  \text{ for } m=5;\\
& \det(-G)=\frac{168\sqrt{3}+291}{64} \approx 9.0935 \text{ for } m=6;\\
\end{align*}
\end{ex}
\begin{ex} { \it The pairwise sum of the inverse distances between the vertices}

Consider $g(x)=(2\sin \frac{x}{2})^{-1}$  and construct  kernel $ f $ using  formula (\ref{f24}). It is easy to see that the constructed kernel $ f $ is equal to $ f(U_1, \ldots, U_m) = \sum_{j> i} \frac{1}{| U_iU_j |}, $ that is, the sum of the inverse distances between the vertices $ U_1, \ldots, U_m. $ This example was already considered by Toth  \cite{n25}.

The minimal value of  kernel $ f $ is attained only at the vertices of the regular polygon  and is equal to $\frac{m}{4} \sum_{k=1}^{m-1}\csc\frac{\pi k}{m}$ (see \cite{n25}).  The second derivative of  function $ g $ is equal to  $g''(x)=\frac{2-\sin^2{\frac{x}{2}}}{8\sin^3{\frac{x}{2}}},$ therefore the elements of the matrix $G$  defined in (\ref{f29}) have the following form:
 \begin{align*}
g_{i,j}=\begin{cases}
-\frac{2-\sin^2{\left(\frac{\pi |i-j|}{m}\right)}}{8\sin^3{\left(\frac{\pi |i-j|}{m}\right)}}, \text{ if } i \ne j,\\
\sum\limits_{s=1}^{m-1} \frac{2-\sin^2{\left(\frac{\pi s}{m}\right)}}{8\sin^3{\left(\frac{\pi s}{m}\right)}}\text{ if } i=j.
\end{cases}
 \end{align*}
 As in Example \ref{exx8}, we could not find a general formula for any $ m $, but again we provide some calculations for the small values of $m$. The value of $ M $ is equal to
 \begin{align*}
 &M=\sqrt{3} \approx 1.73205 \text{ for } m=3,\\
  &M= 2\sqrt{2}+1\approx 3.828427\text{ for } m=4,\\
 &M = \frac{10}{\sqrt{10-2\sqrt{5}}}+\frac{10}{\sqrt{10+2\sqrt{5}}} \approx 6.88191 \text{ for } m=5,\\
    &M=\frac{15}{2}+2\sqrt{3} \approx 10.96410 \text{ for } m=6.\\
 \end{align*}
The value of $ \det(G) $ in this example is equal to
  \begin{align*}
 &\det(G)=\frac{25}{144} \approx 0.17361 \text{ for } m=3,\\
  &\det(G)=\frac{57}{128}\sqrt{2}+\frac{9}{32} \approx 0.911017 \text{ for } m=4,\\
   &\det(G)=\frac{21847+7395\sqrt{5}}{3200} \approx 11.9946\text{ for } m=5,\\
    &\det(G)= \frac{2486141}{13824}+\frac{2224445\sqrt{3}}{27648}  \approx 319.19601 \text{ for } m=6.\\
    \end{align*}
\end{ex}

\section{Proof of the general theorem: part 1}
\label{doc}

We return to the proof of our general Theorem \ref{t2}. It can be  divided into 2 parts. The first part takes the form of the following statement.
\begin{theorem}
\label{t3}
Suppose that  kernel $f$ and  points $U_1, \ldots, U_n$ satisfy Conditions  \textup{A} and \textup{B} from  Section  \ref{tt1}. Let $H_n$ be the $U$-$\max$ statistics constructed via kernel $f$, that is,
$ H_n = \max\limits_{1 \le i_1< \ldots <i_m \le n} f(U_{i_1}, \ldots, U_{i_m}).$
Then the following two statements are true:
\begin{enumerate}
\item{ There exist  constants $C$ and $D $ such that, for any number $\varepsilon$,  $0<\varepsilon<C,$ if \\$f(U_1, \ldots, U_m) \ge M-\varepsilon,$ then
$\min\limits_{1\le i \le d}{\|V_i-\beta\|}\le D\sqrt{\varepsilon},$ where $ \beta $ is defined by  \eqref{bet1} and \eqref{bb1}, and $V_i$ is defined in Condition \textup{A4.}}

\item{The following relation holds true: $$\lim_{\varepsilon \rightarrow 0+} \varepsilon^{-\frac{m-1}{2}} \mathbb{P}\lbrace f(U_1, \ldots, U_m) \ge M-\varepsilon \rbrace = K,$$ where
    $K=\frac{\left(2\pi\right)^{\frac{m-1}{2}}}{\Gamma\left(\frac{m+1}{2}\right)}\sum\limits_{i=1}^k \left( \frac{1}{\det(-G_i)} \int\limits_0^{2\pi} \left(p(x) \prod\limits_{l=1}^{m-1} p\left(x+V_i^l\right) \right)\, dx\right).$}
\end{enumerate}
\end{theorem}

\begin{proof}
It is clear that
$$\mathbb{P}\{f(U_1, \ldots, U_m)>z\} = \mathbb{P}\{h(\beta_1, \ldots, \beta_{m-1})>z\},$$
where $ \beta_i $ are random angles defined in $ \eqref{tt1}. $
Further we  deal with   function $ h $ only.

 Let us define for every $ \varepsilon> 0 $ the  number
 \begin{align}
 \label{ss}
 \hat{S}\left(\varepsilon\right) = \min\lbrace\, s \ge 0 \mid \forall x \in [0,2\pi]^{m-1}: M-f(x) \le \varepsilon \Rightarrow \exists i: \left\|x-V_i\right\| \le s \, \rbrace.
 \end{align}
In other words, $ \hat{S}(\varepsilon) $ is   minimal radius of balls with centers in $ V_i, i = 1,\ldots, k, $ such that  if the value of  function  differs from the maximum value by less than $ \varepsilon, $ then the argument of the  function must lie in one of the balls.
For any $ \varepsilon> 0 $, this minimal radius obviously exists.
Also let us define for any $\varepsilon>0$
\begin{align}
\label{ss1}
S(\varepsilon)=\max(\hat{S}(\varepsilon), \varepsilon^{\frac{1}{3}}).
\end{align}
 It is easy to show that
\begin{align}
 \label{lemm1}
 \lim\limits_{\varepsilon \rightarrow +0}S\left(\varepsilon\right) =0.
 \end{align}
 Indeed,  \eqref{lemm1} is equivalent to the limit relation  $ \lim\limits_{\varepsilon \rightarrow +0}\hat{S}\left(\varepsilon\right) =0. $  Function $ \hat{S}(x) $ is  non-decreasing and  non-negative, therefore the limit $\lim\limits_{\varepsilon \rightarrow +0}\hat{S}(\varepsilon)$ exists and is non-negative. We denote it by $a$ and suppose that $a>0$ (otherwise \eqref{lemm1}  is proved).	
Then for any $\varepsilon>0$  there exists $x_{\varepsilon} \in [0,2\pi]^{m-1}$ such that $\min_{1 \le i \le k}\|V_i-x_{\varepsilon}\| \ge a $ and $|h(x_{\varepsilon})-M|\le \varepsilon.$ Let $\varepsilon_n = \frac{1}{n}.$ Then the infinite sequence $x_{\varepsilon_n}, $  which belongs to the compact set $[0,2\pi]^{m-1}, $ has a convergent subsequence with some limit   $x^{*}.$
 By construction, $x^{*} \ne V_i\text{ for any } i,$  and the continuity of $h $ implies that $ h(x^{*}) = M. $ This contradiction proves \eqref{lemm1} .

Equality  \eqref{lemm1} implies that,  for sufficiently small $ \varepsilon,$  $ S (\varepsilon)$-neighborhoods of points $ V_1, \ldots, V_k $ have empty intersection.
Hence, by definitions  \eqref{ss} and \eqref{ss1} for sufficiently small $ \varepsilon $ the following equality is valid:
\begin{equation}
\label{sum}
\mathbb{P}\{h(\beta)\ge M-\varepsilon\} =\sum_{i=1}^k \mathbb{P}\{h(\beta) \ge M-\varepsilon, \|V_i-\beta\|\le S(\varepsilon) \}.
\end{equation}

Let us fix some $i \in \{1, \ldots, k\}.$  Assume that the following event happens for some  $\varepsilon>0:$
\begin{align}
\label{f4}
h(\beta)=h(\beta_1, \ldots, \beta_{m-1}) \ge M-\varepsilon,\, \|V_i-\beta\|\le S(\varepsilon).
\end{align}

By  \eqref{lemm1} there exists $ \varepsilon_0> 0 $ such that $ S(\varepsilon_0) <\frac{\delta}{2}, $ where $ \delta $ is the  number from Condition A5. Function $ S(\varepsilon) $ is  non-decreasing, therefore for any positive $ \varepsilon<\varepsilon_0 $ we have
\begin{align}
\label{f7}
S(\varepsilon)<\frac{\delta}{2}.
\end{align}
Below we deal with $ \varepsilon<\varepsilon_0 $ only.
Since  function $ h $ is three times continuously differentiable in the $ \delta$-neighborhood of any maximal point, in this neighborhood we consider the Taylor expansion of  function $h$ at the point $ V_i $ with the third order remainder. For this purpose we introduce the  notation:
\begin{align}
\label{f17}
\alpha_j = \beta_j-V_i^j \text{ and } \alpha=(\alpha_1, \ldots, \alpha_{m-1}).
\end{align}
It is clear that $$\left\|\alpha\right\|=\left\|\beta-V_i\right\|  < \frac{\delta}{2}.$$
Here  $ \alpha $  is an element of $ \mathbb {R}^{m-1} $ which is considered as a difference of two elements of $ \mathbb{R}^{m-1} $ and not as the difference of two sets of angles.
  By  \eqref{lemm1} and Condition  A4 it is the same for small $ \varepsilon. $

We write the Taylor expansion of  function $h$  at the point $ V_i.$
Then we have
\begin{align}
\nonumber
&h\left(\beta\right)=h\left(\beta_1, \ldots, \beta_{m-1}\right)=h\left(V_i^1+\alpha_1, V_i^2+\alpha_2, \ldots, V_i^{m-1}+\alpha_{m-1}\right)\\
\label{f6}
&=h\left(V_i\right)+\sum_{j=1}^{m-1}  \frac{\partial h\left(V_i\right)}{\partial x_j}\, \alpha_j + \sum_{1 \le l,s \le m}\frac{1}{2} \frac{\partial h\left(V_i\right)}{\partial x_l\partial x_s} \, \alpha_l\alpha_s
\\&+\sum_{1 \le l,s,t \le m}\frac{1}{6} \frac{\partial h\left(V_i+r_{\left(l,s,t\right)}\right)}{\partial x_l \partial x_s \partial x_t}\, \alpha_l \alpha_s \alpha_t,
\nonumber
\end{align}
where $r_{\left(l,s,t\right)} =c_{\left(l,s,t\right)} \cdot \left(\alpha_1, \ldots,\alpha_{m-1}\right),$ and $c_{\left(l,s,t\right)} \in \left(0,1\right)$ are constants depending on indices $l,s,t$ and on function $h.$ According Condition A4, $ V_i $ does not lie on the boundary of the definition domain  of  the continuous function $ h, $ therefore $\frac{\partial h(V_i)}{\partial x_j} =0$ for all $ j \in \{1, \ldots, m-1 \}. $ Hence, the linear term in  expansion (\ref{f6}) is equal to 0.

Consider the matrix
\begin{align}
\label{f5}
 A^i=\frac{1}{2}G_i,
\end{align} where $G_i$ are the same as in Condition  A6.
It is clear that the coefficient before $ \alpha_l \alpha_s $ in  (\ref{f6}) is $ a^i_{l, s} $
(the element of the matrix $ A^i$).Thus,
\begin{equation}
\label{f1}
\begin{split}
&h\left(\beta\right)=M+\sum\limits_{1 \le l,s \le m}a^i_{l,s} \alpha_l\alpha_s+\frac{1}{6}\sum_{1 \le l,s,t \le m} \frac{\partial h\left(V_i+r_{(l,s,t)}\right)}{\partial x_l \partial x_s \partial x_t} \alpha_l \alpha_s \alpha_t.
\end{split}
\end{equation}
Therefore,  condition (\ref{f4}) is equivalent to
\begin{equation}
\label{f2}
-\sum_{1 \le l,s \le m}a^i_{l,s} \alpha_l\alpha_s-\frac{1}{6}\sum_{1 \le l,s,t \le m} \frac{\partial h(V_i+r_{(l,s,t)})}{\partial x_l \partial x_s \partial x_t} \alpha_l \alpha_s \alpha_t \le \varepsilon.
\end{equation}
Under  conditions $ (\ref{f4}) $ and $ (\ref{f7}) $, we estimate the  third order terms  in this formula. Since  functions   $\frac{\partial h(V_i+r)}{\partial x_l \partial x_s \partial x_t} $ are continuous for $|r| \le \frac{\delta}{2},$  there exists $ M_1 $ such that  $\left|\frac{\partial h(V_i+r)}{\partial x_l \partial x_s \partial x_t}\right|$ does not exceed $ M_1 $ for all $ |r| \le \frac{\delta}{2}. $
Therefore, the following inequality holds true:
\begin{align}
&\left|\frac{1}{6} \frac{\partial h\left(V_i+r_{\left(l,s,t\right)}\right)}{\partial x_l \partial x_s \partial x_t} \alpha_l \alpha_s \alpha_t\right| \le M_1 \left|\alpha_l \alpha_s \alpha_t\right|
\label{f9}
\\
& \le M_1 \frac{\left|\alpha_l\right|^3+\left|\alpha_s\right|^3+\left|\alpha_t\right|^3}{3} \le M_1 \frac{\alpha_l^2+\alpha_s^2+\alpha_t^2}{3} S(\varepsilon).
\nonumber
\end{align}
The last inequality follows from  $ \| \alpha \|~ <~S(\varepsilon). $
Summing $ (\ref{f9}) $ over all triples  $ (l, s, t), $ we get the inequality:
\begin{equation}
\label{f10}
\left|\frac{1}{6}\sum_{1 \le l,s,t \le m} \frac{\partial h(V_i+r_{(l,s,t)})}{\partial x_l \partial x_s \partial x_t} \alpha_l \alpha_s \alpha_t\right| \le m^2M_1 S(\varepsilon)\sum_{s=1}^{m-1}\alpha_s^2 =M_2S(\varepsilon) \sum_{s=1}^{m-1}\alpha_s^2.
\end{equation}
Therefore,  the following estimate is valid for all $\|\alpha\|<S(\varepsilon)$:
\begin{align}
\label{f3}
&M_2S\left(\varepsilon\right) \sum\limits_{s=1}^{m-1}\alpha_s^2 -\sum\limits_{1 \le l,s \le m}a^i_{l,s} \,\alpha_l\alpha_s \ge
M-h\left(V_i+\alpha\right) \\
\nonumber
 &\ge -M_2S\left(\varepsilon\right) \sum\limits_{s=1}^{m-1}\alpha_s^2 -\sum\limits_{1 \le l,s \le m}a^i_{l,s}\, \alpha_l\alpha_s.
\end{align}
Collecting the results of  (\ref{f4}), (\ref{f2}) and (\ref{f3}), we obtain
\begin{align}
\nonumber
&\mathbb{P}\Big \lbrace \left(-\sum\limits_{1 \le l,s \le m-1}a^i_{l,s} \alpha_l \alpha_s -M_2 S\left(\varepsilon\right)\sum\limits_{s=1}^{m-1}\alpha_s^2\right) \le \varepsilon, \|\alpha\|<S(\varepsilon) \Big \rbrace \\
\label{f11}
& \ge  \mathbb{P}\lbrace h\left(\beta\right) >M-\varepsilon, \left\|V_i-\beta\right\| \le S\left(\varepsilon\right)\rbrace  \\
\nonumber
&\ge \mathbb{P}\Big\lbrace \left(-\sum\limits_{1 \le l,s \le m-1}a^i_{l,s}\alpha_l \alpha_s +M_2 S\left(\varepsilon\right)\sum\limits_{s=1}^{m-1}\alpha_s^2 \right)\le \varepsilon, \|\alpha\|<S(\varepsilon) \Big\rbrace.
\end{align}
Denote
\begin{align}
\label{f12}
A^i \left(\varepsilon\right)=
\begin{cases}
&A^i+M_2 S\left(\varepsilon\right)I_{m-1}, \text{ for } \varepsilon \ge 0,  \\
&A^i -M_2 S\left(-\varepsilon\right)I_{m-1}, \text{ for } \varepsilon \le 0,
\end{cases}
\end{align}
where $ A^{i} $ is the same as in   $ (\ref{f5}), $ and $ I_{m-1} $ is the identity matrix of size $ (m-1)\times (m-1). $
Then  inequality $ (\ref{f11}) $ may be rewritten using the scalar product $ \langle \, \cdot\, ,\, \cdot\, \rangle $ as
\begin{align}
\label{pr}
&\mathbb{P}\lbrace -\langle A^i\left(\,\varepsilon\right)\alpha, \alpha \rangle \le \varepsilon, \|\alpha\|<S(\varepsilon)\,\rbrace \ge \mathbb{P}\lbrace h\left(\beta\right) \ge M-\varepsilon, \left\|V_i-\beta\right\|\le S\left(\varepsilon\right)\rbrace  \\
\nonumber
&\ge \mathbb{P}\lbrace \, -\langle A^i\left(-\varepsilon\right)\alpha, \alpha \rangle \le \varepsilon, \|\alpha\|<S(\varepsilon)\,\rbrace.
\end{align}
Next, we  need the following lemma.
\begin{lemma}
\label{ll2}
There exist  constants $ \varepsilon_1> 0 $ and $ \lambda> 0 $ such that for any $ | \varepsilon | <\varepsilon_1 $ the matrix $ A^i(\varepsilon) $ is negative definite, and all its eigenvalues do not exceed $ - \lambda. $
\end{lemma}
\begin{proof}

It is  well known that a symmetric real valued matrix $ B $ is negative definite iff all the eigenvalues of the matrix $ B $ are negative (see, e.g., \cite[ \S. 4.1, p. 231, Th. 4.1.10]{MA}). Therefore, it is enough to prove that all eigenvalues of matrix $A^i(\varepsilon)$ do not exceed some $-\gamma<0$ for small $\varepsilon$ to prove Lemma \ref{ll2}.
 Denote by
\begin{align}
\label{oe}
\lambda_1(B) \le \ldots  \le \lambda_{s}(B)
\end{align} the eigenvalues of Hermitian matrix $B$ of  size $s \times s$.
The matrix $ A^i $ is negative semidefinite (see, e.g.,  \cite[ \S 4.5, p. 563]{n15}).   By Condition A6 and  $(\ref{f5})$   $\det\left(A^i\right) \ne 0, $ hence, the matrix $A^i$ is negative definite. Using notation  \eqref{oe} it is equivalent to $\lambda_{m-1}(A^i)<0.$

In what follows, we will need the Weyl theorem  formulated below. It may be found, e.g., in  \cite[\S 4.3, p. 239, Th. 4.3.1]{MA}.
\begin{theorem*}[Weyl]
Let $A,B $ be Hermitian matrices of size $s \times s$. Then using  notation  \eqref{oe} we have $$\lambda_i(A+B) \le \lambda_{i+j}(A)+\lambda_{s-j}(B) $$ for all $i \in \{1, \ldots, s\}, \, j \in \{0, \ldots, s-i\}.$
\end{theorem*}

Due to the Weyl theorem the eigenvalues of $A^i(\varepsilon)$ do not exceed $\lambda_{m-1}(A^i)+M_2S\left(\left|\varepsilon\right|\right).$ Therefore, we can put $\gamma=\frac{|\lambda_{m-1}(A^i)|}{2}$ and $\varepsilon_0$ such that  $S(\varepsilon_0)< \frac{|\lambda_{m-1}(A^i)|}{2M_2}.$ By \eqref{lemm1} such $\gamma$ and $\varepsilon_0$ satisfy the conditions of Lemma \ref{ll2}.
\end{proof}

\begin{rem}
Obviously, $$\lim\limits_{\varepsilon \rightarrow 0} \det(A^i(\varepsilon)) = \det(A^i). $$
\end{rem}

\begin{lemma}
\label{ll3}
There exist constants $ C, D> 0 $ such that if for some $ \varepsilon> 0 $ and $ \beta $ the following conditions are satisfied:
\begin{enumerate}
\item $h(\beta) > M - \varepsilon,$
\item $\|\alpha\|=\|V_i-\beta\|<S(\varepsilon),$
\item $\varepsilon<C,$
\end{enumerate}
then $ \| \alpha \| <D \sqrt{\varepsilon}. $
The notation $ \alpha $ and $ \beta $ are the same as above.
\end{lemma}
\begin{proof}
Denote by $ C = \min(\varepsilon_1, \varepsilon_0), $ where $ \varepsilon_1 $ is from Lemma \ref{ll2} and $ \varepsilon_0 $ is from  $ (\ref{f7}).$
According to the Rayleigh theorem (see, e.g., \cite[\S 4.2, p. 234, Th.4.2.2]{MA}) the inequality $\langle Bx,x \rangle  \le \lambda_s(B) \|x\|^2,$ where $\lambda_s(B)$ is from \eqref{oe},  holds for every Hermitian matrix $B$ of  size $s \times s$.
 Then for  any $ | \varepsilon | <C $ we have  $ \langle A^i (\varepsilon) x, x \rangle \le  -\gamma \| x \|^2, $ where $ \gamma $ is also from Lemma \ref{ll2}. Note that by  $ (\ref{f3}) $ the following inequality is valid for positive $\varepsilon$: $$ M + \langle A^i (\varepsilon) \alpha, \alpha \rangle \ge h(V_i + \alpha) > M - \varepsilon. $$ Therefore, $ \varepsilon \ge -\langle A^i(\varepsilon)\alpha,\alpha\rangle \ge \gamma\|\alpha\|^2.$
Hence, $ \| \alpha \|$ is less than $\frac{\sqrt{\varepsilon}}{\sqrt{ \gamma}}. $
\end{proof}
The following Corollary follows from  the proof of Lemma \ref{ll3}.
\begin{sled}
\label{sl3}
If $ \varepsilon \ge - \langle A^i (\varepsilon) \alpha, \alpha \rangle $ and $ 0 <\varepsilon <C, $ then $ \| \alpha \| \le D \sqrt{\varepsilon}. $
\end{sled}

Lemma \ref{ll3} completely proves the first assertion of Theorem \ref{t3}.
It remains to prove the second one. For this purpose we need to calculate $$\mathbb{P}\lbrace\langle-A^i\left(\pm\varepsilon\right) \alpha, \alpha\rangle \le \varepsilon, \|\alpha\|<S(\varepsilon) \rbrace$$ for small $ \varepsilon. $
By Corollary  \ref{sl3} condition $\|\alpha\|<S(\varepsilon),$ where $S(\varepsilon)$ is defined in  \eqref{lemm1}, follows from the inequality $\langle -A^i\left(\pm\varepsilon\right) \alpha, \alpha\rangle \le \varepsilon$ for sufficiently small $\varepsilon.$ Therefore, $$\mathbb{P}\lbrace\langle-A^i\left(\pm\varepsilon\right) \alpha, \alpha\rangle \le \varepsilon, \|\alpha\|<S(\varepsilon) \rbrace = \mathbb{P}\lbrace\langle-A^i\left(\pm\varepsilon\right) \alpha, \alpha\rangle \le \varepsilon \rbrace.$$

Below we assume  that $ 0 <\varepsilon <C, $ where $ C $ is the same as in Lemma \ref{ll3}.
\begin{lemma}
\label{ll4}
For sufficiently small $ \varepsilon $ there exists  $x^{*}(\varepsilon)\in [0,2\pi]^{m-1}$ such that  following equality holds:
\begin{align*}
\mathbb{P}\lbrace\langle-A^i\left(\pm\varepsilon\right) \alpha, \alpha\rangle \le \varepsilon \rbrace = \frac{\left(\pi \varepsilon\right)^{\frac{m-1}{2}}}{\Gamma\left(\frac{m+1}{2}\right)\sqrt{\det(-A^i(\varepsilon))}} \cdot \int\limits_0^{2\pi} \left[p(y) \prod\limits_{l=1}^{m-1} p(y +V_i^l+x^{*}_l(\varepsilon)) \right]\, dy,
\end{align*}
where $\|x^{*}(\varepsilon)\|\le D\sqrt{\varepsilon},$ and  constant  $D$ is introduced in Lemma \textup{\ref{ll3}.}
\end{lemma}
\begin{proof}
In order to simplify the formulas, we put $B=-A^i(\varepsilon)=(b_{i,j})_{i, j=1}^{m-1, m-1}.$  We also introduce $ \beta_0 = \angle xOU_1 $ to be the angle between the axis $ Ox $ and the vector $ OU_1, $ taken counterclockwise.
We have
\begin{align}
\nonumber
&\mathbb{P}\lbrace\langle-A^i\left(\pm\varepsilon\right) \alpha, \alpha\rangle \le \varepsilon \rbrace=\mathbb{P} \{\langle B\alpha, \alpha\rangle \le \varepsilon \} =\mathbb{P}\Big\lbrace\sum\limits_{1 \le s,t \le m-1} b_{s, t} \alpha_s \alpha_t \le \varepsilon \Big \rbrace \\
&=\int\limits_{0}^{2\pi} p(y) \int_{\mathbb{R}^{m-1}} {\bf 1} \Big\lbrace \sum\limits_{1 \le s,t \le m-1} b_{s, t} x_s x_t \le\varepsilon \Big \rbrace \prod_{l=1}^{m-1}\rho_l(x_l|\beta_0=y) \, dx_1 \cdots dx_{m-1} dy,
\label{f13}
\end{align}
where $\rho_l(x_l|\beta_0=y)$  is the conditional density of $ \alpha_l=x_l $  given $\beta_0=y. $ Taking into account that $ U_1, \ldots, U_m $ are independent random variables and using   $ (\ref{f17}) $ we obtain that $\rho_l(x_l|\beta_0=y)=p(y+V_i^l+x_l).$

Also note that in the general case the second integral is taken not over $ \mathbb{R}^{m-1} $ but over quotient space $ \mathbb{R}^{m-1}/_\sim, $ where $ x, y \in \mathbb{R}^{m-1} $ and $ x \sim y, $ if for any $ i \in \{1, \ldots, m-1 \} $ we have $ x^i- y^i = 2 \pi r $ for some $ r \in \mathbb{Z}. $ The reason  is that the values of $\alpha$ which differ by $2\pi r$ correspond to the same angle $\beta.$ But by Corollary  $ \ref{sl3} $, inequality $ \langle B \alpha, \alpha \rangle <\varepsilon $ implies  $ \| \alpha \| <D \sqrt{\varepsilon}. $ Therefore, for small $ \varepsilon, $ all the $ \alpha \in \mathbb{R}^{m-1} $ satisfying this inequality  correspond to different $ \beta. $
 Hence, for sufficiently small $ \varepsilon, $ we  may integrate indeed over $ \mathbb{R}^{m-1}. $ Using these facts, we continue  equalities (\ref{f13}):
\begin{align}
&\nonumber \mathbb{P}\lbrace\langle-A^i\left(\pm\varepsilon\right) \alpha, \alpha\rangle \le \varepsilon \rbrace\\
&\nonumber =\int\limits_{0}^{2\pi} p(y) \int_{\mathbb{R}^{m-1}} {\bf 1} \Big \lbrace \sum\limits_{1 \le s,t \le m-1} b_{s, t} x_s x_t \le\varepsilon \Big \rbrace \prod_{l=1}^{m-1}p(y+V_i^l+x_l)\, dx_1 \cdots dx_{m-1} \, dy\\
&=\int_{\mathbb{R}^{m-1}} {\bf 1} \Big\lbrace \sum\limits_{1 \le s,t \le m-1} b_{s, t} x_s x_t \le\varepsilon \Big \rbrace\int\limits_{0}^{2\pi} p(y)  \prod_{l=1}^{m-1}p(y+V_i^l+x_l)\, dy \, dx_1 \cdots dx_{m-1}.
  \label{f14}
\end{align}
This formula is  obtained just by switching the   integration order. In order to calculate this integral, we define the set
\begin{equation}
 \label{set}
 T=\Big\lbrace x \in \mathbb{R}^{m-1}:  \sum_{1 \le s,t \le m-1} b_{s, t} x_s x_t \le \varepsilon \Big \rbrace= \lbrace x \in \mathbb{R}^{m-1}: \langle x, Bx \rangle \le \varepsilon  \rbrace.
 \end{equation}
  The set $T$ is an ellipsoid with center 0 and configuration matrix $(\frac{1}{\varepsilon}\cdot B)^{-1},$ see \cite[p. 97]{Ell}. It is well-known, that  the ellipsoid  with center 0 and symmetric positive semidefinite configuration matrix  $Q$ is defined by  $\{x \in \mathbb{R}^{p}: \langle x, Q^{-1}x \rangle \le 1\}$ and its volume is equal to $\frac{\pi^{\frac{p}{2}}}{\Gamma(\frac{p}{2}+1)\sqrt{\det{Q}}}.$ It is mentioned, e.g., in \cite[p. 103]{Ell}.
Therefore, the  Lebesgue measure of the set $ T $  satisfies the equality
\begin{align}
 \label{f18}
 \textup{mes}(T)=\frac{\left(\varepsilon \pi\right)^{\frac{m-1}{2}}}{\sqrt{\det(B)}\Gamma\left(\frac{m+1}{2}\right)}.
 \end{align}

In what follows, we will need a mean value theorem  formulated below. It may be found, e.g., in  \cite{n23}.
\begin{theorem*}[Mean value theorem]
Let $E \subset \mathbb{R}^n$ be a connected set of finite measure. If  function $ f $ is continuous and summable on $ E, $ then there exists $ c \in E $ such that $$\int_E f(x) \, dx=f(c)\cdot \textup{mes}(E).$$
\end{theorem*}

Using this theorem, we  obtain the equality:
\begin{align}
&\nonumber \int_{\mathbb{R}^{m-1}} {\bf 1}\Big \lbrace \sum\limits_{1 \le s,t \le m-1} b_{s, t} x_s x_t \le \varepsilon \Big \rbrace \int\limits_{0}^{2\pi} p(y)  \prod_{l=1}^{m-1}p(y+V^i_l+x_l) \, dy \,  dx_1 \cdots dx_{m-1}\\
 & \label{f15}=\int\limits_{0}^{2\pi} p(y)  \prod_{l=1}^{m-1}p(y+V^i_l+x^{*}_l(\varepsilon))\, dy \cdot \textup{mes}\left(T\right), \text{ where } x^{*}(\varepsilon) \in T.
\end{align}
Recall that the set  $ T $ defined in (\ref{set}). The condition $x^{*}(\varepsilon) \in T$ together with  (\ref{set}) implies $ \sum\limits_{1 \le s,t \le m-1} b_{s, t} x^{*}_s(\varepsilon) x^{*}_t(\varepsilon) \le \varepsilon.$ Therefore, Corollary \ref{sl3} implies   inequality $\|x^{*}(\varepsilon)\|<D\sqrt{\varepsilon},$  where the constant $ D $ is  introduced in Lemma \ref{ll3}. Hence, collecting formulas (\ref{f14}), \eqref{f18}, (\ref{f15}) , we  finish the proof.
\end{proof}

Substituting the result of Lemma \ref{ll4} into inequality $ (\ref{pr}), $ we get  for small $ \varepsilon $ the relation
\begin{align*}
&\frac{\left(\pi \varepsilon\right)^{\frac{m-1}{2}}}{\Gamma\left(\frac{m+1}{2}\right)\sqrt{\det(-A^i(\varepsilon))}} \cdot \int\limits_0^{2\pi} \left[p(y) \prod\limits_{l=1}^{m-1} p\left(y+V_i^l+x^{*}_l(\varepsilon)\right) \right]\, dy \\
&\ge \mathbb{P}\{h(\beta) \ge M-\varepsilon, \|V_i-\beta\| \le S(\varepsilon )\}  \\
& \ge\frac{\left(\pi \varepsilon\right)^{\frac{m-1}{2}}}{\Gamma\left(\frac{m+1}{2}\right)\sqrt{\det(-A^i(-\varepsilon))}} \cdot \int\limits_0^{2\pi} \left[p(y) \prod\limits_{l=1}^{m-1} p\left(y+V_i^l+x^{*}_l(-\varepsilon)\right) \right]\, dy.
\end{align*}

Dividing both sides of this inequality by $ \varepsilon^{\frac{m-1}{2}} $ and tending $ \varepsilon $ to 0,  we get
\begin{align}
&\lim_{\varepsilon \rightarrow +0}\frac{\pi^{\frac{m-1}{2}}}{\Gamma\left(\frac{m+1}{2}\right)\sqrt{\det(-A^i(\varepsilon))}} \cdot \int\limits_0^{2\pi} \left[p(y) \prod\limits_{l=1}^{m-1} p\left(y+V_i^l+x^{*}_l(\varepsilon)\right) \right]\, dy
\label{f30}\\
&\nonumber \ge \lim_{\varepsilon \rightarrow +0}\varepsilon^{-\frac{m-1}{2}}\mathbb{P}\{h(\beta) \ge M-\varepsilon, \|V_i-\beta\|\le S(\varepsilon)\}  \\
&\nonumber \ge \lim_{\varepsilon \rightarrow +0} \frac{\pi^{\frac{m-1}{2}}}{\Gamma\left(\frac{m+1}{2}\right)\sqrt{\det(-A^i(-\varepsilon))}} \cdot \int\limits_0^{2\pi} \left[p(y) \prod\limits_{l=1}^{m-1} p\left(y+V_i^l+x^{*}_l(-\varepsilon)\right) \right]\, d\beta_0.
\end{align}
Using that
\begin{align*}
&\lim_{\varepsilon \rightarrow 0} \det(A^i(\varepsilon)) =\det(A^i),\\ &\lim_{\varepsilon \rightarrow 0} \int\limits_0^{2\pi} \left[p(y) \prod\limits_{l=1}^{m-1} p\left(y+V_i^l+x^{*}_l(\varepsilon)\right) \right]\, dy =  \int\limits_0^{2\pi} \left[p(y) \prod\limits_{l=1}^{m-1} p\left(y+V_i^l\right) \right] \,dy,
\end{align*}
 we pass to the limit in \eqref{f30} and obtain
\begin{align}
\label{eq}
&\lim_{\varepsilon \rightarrow +0}\varepsilon^{-\frac{m-1}{2}}\mathbb{P}\{h(\beta) \ge M-\varepsilon, \|V_i-\beta\| \le S(\varepsilon)\} \\
\nonumber
 &=\frac{\pi^{\frac{m-1}{2}}}{\Gamma\left(\frac{m+1}{2}\right)\sqrt{\det(-A^i)}} \cdot \int\limits_0^{2\pi} \left[p(y) \prod\limits_{l=1}^{m-1} p\left(y+V_i^l\right) \right]\, dy.
\end{align}
From  $ (\ref{f5}) $ we may conclude that \begin{align}
\label{f16}
\det(-A^i)=\frac{\det(-G_i)}{2^{m-1}}.
\end{align}
Now, using  $ (\ref{eq}), (\ref{sum}) $ and $ (\ref{f16}), $ we obtain
\begin{align*}
&\lim_{\varepsilon \rightarrow 0+}\varepsilon^{-\frac{m-1}{2}}\mathbb{P}\{h(\beta) \ge M-\varepsilon\} =  \frac{\pi^{\frac{m-1}{2}}}{\Gamma\left(\frac{m+1}{2}\right)}\sum\limits_{i=1}^k \frac{\int\limits_0^{2\pi} \left[p(y) \prod\limits_{l=1}^{m-1} p\left(y+V_i^l\right) \right] \,dy}{\sqrt{\det(-A^i)}} \\
&=\frac{\left(2\pi\right)^{\frac{m-1}{2}}}{\Gamma\left(\frac{m+1}{2}\right)}\sum\limits_{i=1}^k \frac{\int\limits_0^{2\pi} \left[p(y) \prod\limits_{l=1}^{m-1} p\left(y+V_i^l\right) \right]\, dy}{\sqrt{\det(-G_i)}},
\end{align*}
and Theorem \ref{t3} is proved.
\end{proof}

\section{Proof of the general theorem: part 2}
\label{cont}

In this subsection, we prove the second part of our general theorem.
\begin{theorem}
\label{t4}
Suppose that  kernel $f$ and  points $U_1, \ldots, U_n$ satisfy Conditions  \textup{A} and \textup{B} which are listed in Section \ \ref{tt1}.
Let $H_n$ be  $U$-$\max$ statistics constructed via kernel $f$, that is,
$ H_n = \max\limits_{1 \le i_1< \ldots <i_m \le n} f(U_{i_1}, \ldots, U_{i_m}).$
Then for every  $t>0$ the following relation holds true:
\begin{equation}
\label{mm}
\lim_{n \rightarrow \infty} \mathbb{P}\lbrace n^{\frac{2m}{m-1}} (M-H_n)\le t)\rbrace=1 -e^{-\frac{t^{\frac{m-1}{2}}K}{m!}},
\end{equation}
where the constant $ K $ is introduced in condition 2 of Theorem \textup{\ref{t3}}.
The rate of convergence is $O(n^{-\frac{1}{m-1}})$ for  $m>1$ and   $O(n^{-1})$
for $m=1$.
\end{theorem}

\begin{proof}
For any $ t> 0 $ we define the transformation$$ z_n(t)=M-tn^{-\frac{2m}{m-1}} .$$
Let us consider $\lambda_{n, z_n(t)}$ defined in Theorem \ref{LM}. Then
$$\lambda_{n,z_n(t)} =\frac{n!}{m! (n-m)!}\mathbb{P}\{f(U_1, \ldots, U_m)>z_n(t)\}.$$
We put $\varepsilon=tn^{-\frac{2m}{m-1}}, $ then $n^m \varepsilon^{\frac{m-1}{2}} = t^{\frac{m-1}{2}}.$
Let us prove the fulfillment of Condition (\ref{fff1}) of  Theorem \ref{SB} (Silverman–-Brown Theorem). We write:
\begin{align*}
&\lim_{n \rightarrow \infty} \lambda_{n,z_n(t)} = \lim_{n \rightarrow \infty} \frac{n!}{m! (n-m)!}\mathbb{P}\{f(U_1, \ldots, U_m)>z_n(t)\} \\
&=\frac{1}{m!} \lim_{n \rightarrow \infty} \frac{n!}{n^m (n-m)!} n^m \varepsilon^{\frac{m-1}{2}} \varepsilon^{-\frac{m-1}{2}} \mathbb{P}\{f(U_1, \ldots, U_m)>\varepsilon\} \\
&=\frac{1}{m!} t^{\frac{m-1}{2}} \lim_{n \rightarrow \infty} (tn^{-\frac{2m}{m-1}})^{-\frac{m-1}{2}}\mathbb{P}\{f(U_1, \ldots, U_m)>M-tn^{-\frac{2m}{m-1}}\}\\
&=\frac{t^{\frac{m-1}{2}}K}{m!} = :\lambda_t>0.
\end{align*}

In the last line, we used the second statement of Theorem \ref{t3}. Now we will prove  Condition (\ref{zz1}) of Remark \ref{rem2}, which has the  form:
 $$\lim_{n \rightarrow \infty} n^{2m-r} p_{z_n(t)}\tau_{z_n(t)}(r)=0 \text{ for any } r \in \{1, \ldots, m-1 \}.$$ According to Remark \ref{rem2} Condition (\ref{fff2}) of Theorem \ref{SB} can be replaced by this one.
We formulate this statement as a separate lemma.

\begin{lemma}
\label{ll5}
 For each  $r \in \{1, \ldots, m-1 \}$ we have the following  relation: $$\lim_{n \rightarrow \infty}n^{2m-r}\mathbb{P}\{h(U_1, \ldots, U_m)>z_n(t), h(U_{1+m-r}, \ldots, U_{2m-r})>z_n(t)\}=0.$$
\end{lemma}

\begin{proof}
Let us introduce the following notation:
$\beta_i=\angle U_1OU_{i+1} \text{ for } i \in \{1, \ldots, 2m-r-1\},
\gamma_i=\angle U_{m-r+1}OU_{i+1} \text{ for } i \in \{m-r, \ldots, 2m-r-1\}.$
Such  a notation corresponds to  (\ref{bb1}) and (\ref{bet1}) for  each $i \in \{1, \ldots, m-1\}.$
It is clear that $\gamma_i=(\beta_i-\beta_{m-r}) \mod{2\pi}$ for any  $i \ge m.$
We introduce the  events
 $Q_{i,j}= \{\|V_i- (\beta_1, \ldots, \beta_{m-1})\| \le D\sqrt{\varepsilon}, \|V_j- (\gamma_{m-r}, \ldots, \gamma_{2m-r-1})\| \le D\sqrt{\varepsilon}\},$ where $V_i$ is the same as in Condition  A4   and constant $D$ is introduced in Theorem \ref{t3}.
 It follows from Lemma \ref{ll3}  that  for small $ z_n (t) $  the following equality holds true:
\begin{align}
&\nonumber\{ h(U_1, \ldots, U_m)\ge z_n(t)\cap h(U_{1+m-r}, \ldots, U_{2m-r})\ge z_n(t) \} \\
&=\cup_{ 1\le i,j \le k} \left(\left[h(U_1, \ldots, U_m)\ge z_n(t)\cap h(U_{1+m-r}, \ldots, U_{2m-r})\ge z_n(t)\right] \cap Q_{i,j}\right).
\label{f20}
\end{align}

Next, we estimate the  probability
 \begin{align}
\label{pp1}
\mathbb{P} \{\left[h(U_1, \ldots, U_m)\ge z_n(t)\cap h(U_{1+m-r}, \ldots, U_{2m-r})\ge z_n(t)\right] \cap Q_{i,j} \}.
\end{align}
By the definition  for all elements  $V_i$ from $Q_{i,j}$  we have the following bounds for $\beta_i$ and $\gamma_i:$ $\|\beta_{l}-V_i^l\| \le D\sqrt{\varepsilon}$ for each  $i<m,$ and  $\|\gamma_{l}-V_j^{l-m+r}\|\le D\sqrt{\varepsilon}.$

For $ l \ge m $ we obtain $$\left\|\beta_l-V_i^{m-r}-V_j^{l-m+r}\right\| \le \left\|\beta_l-\beta_{m-r}-V_j^{l-m+r}\right\|+\left\|\beta_{m-r}-V_i^{m-r}\right\| \le 2D\sqrt{\varepsilon}.$$ Denote by $ M_p $ the maximal value of  density $p(x).$  Using the properties of distribution   of $\beta_l,$  we can estimate  the upper bound of probability $(\ref{pp1})$ by
 $\left(2DM_p\sqrt{\varepsilon}\right)^{m-1} \left(4DM_p\sqrt{\varepsilon}\right)^{m-r}.$
Using   formula $(\ref{f20})$ and substituting $\varepsilon=tn^{-\frac{2m}{m-1}}$ in the estimate of (\ref{pp1}) we obtain the  inequality
\begin{multline*}
n^{2m-r}\mathbb{P}\{h(U_1, \ldots, U_m)\ge z_n(t), h(U_{1+m-r}, \ldots, U_{2m-r})\ge z_n(t)\}\\
 \le n^{2m-r}  k^2 \left(2DM_p\sqrt{tn^{-\frac{2m}{m-1}}}\right)^{m-1} \left(4DM_p\sqrt{tn^{-\frac{2m}{m-1}}}\right)^{m-r}=O(n^{-\frac{m-r}{m-1}})=o(1).\,
\end{multline*}
\end{proof}

Let us return to the proof of Theorem \ref{t4}. Now we may use Theorem \ref{SB}, since all its conditions are verified. Then according to  \eqref{fff3} we obtain
$$\lim_{n \rightarrow \infty} \mathbb{P}\left(H_n \le z_n(t)\right)=e^{-\lambda_t}$$
 for any $t \in T.$
Hence, $$\lim_{n \rightarrow \infty} \mathbb{P}\left(H_n \le M-t n^{-\frac{2m}{m-1}}\right)=e^{-\frac{t^{\frac{m-1}{2}} K}{m!}}.$$
Therefore, for any $ t> 0 $ the following relation is valid:
$$\lim_{n \rightarrow \infty} \mathbb{P}\{ n^{\frac{2m}{m-1}} (M-H_n)\le t) \}=1 -e^{-\frac{t^{\frac{m-1}{2}}K}{m!}}.$$

According to Remark \ref{rem2}, the convergence rate is  $O\left(n^{-1}+\sum_{r=1}^{m-1}p_{n,z_n(t)}\tau_{n,z_n(t)}(r)n^{2m-r}\right).$ From the proof of Lemma \ref{ll5} it follows that $O\left(p_{n,z_n(t)}\tau_{n,z_n(t)}(r)n^{2m-r}\right)=O\left(n^{-\frac{m-r}{m-1}}\right),$ therefore it is   $O\left(n^{-\frac{1}{m-1}}\right)$ for $ m> 1, $  and it is $ O\left(n^{- 1}\right)$ for $ m = 1 $
 Hence, Theorem \ref{t4} is proved.
\end{proof}

The combination of Theorems \ref{t3} and \ref{t4} implies our general Theorem \ref{t2}.

\section{Acknowledgements}
The authors would like to thank Dr. A. Yu. Zaitsev and  Dr. D. Zaporozhets  for their
invaluable help concerning this paper.

\bibliography{mybibfile}

\end{document}